\documentclass[12pt,oneside,a4paper]{article}
\usepackage{amsmath,amsthm}
\usepackage{hyperref}
\usepackage{geometry}
\usepackage{url}
\usepackage[utf8]{inputenc}
\usepackage{lmodern}    
\usepackage{ifthen}  
\usepackage[all]{xy}
\usepackage{hyperref}
\usepackage{cleveref}

\theoremstyle{remark}
\newtheorem{remark}{Remark}[section]
\newtheorem*{remark 2}{Historical remark}

\newtheoremstyle{mytheorem}
  {3pt}
  {3pt}
  {\itshape}
  {}
  {\bfseries}
  {.}
  {1em}
  {}
\theoremstyle{mytheorem}
\newtheorem{theorem}[remark]{Theorem}
\newtheorem{theoremA}{Theorem}

\newtheorem{theoremB}{Theorem}

\newtheorem{theoremC}{Theorem}

\newtheorem{theoremD}{Theorem}

\usepackage{sectsty}
\sectionfont{\centering}
\usepackage{secdot}
\newtheorem{definition}[remark]{Definition}
\newtheorem*{definition 2}{Definition (Gromov)}

\newtheorem{proposition}[remark]{Proposition}

\newtheorem{property}[remark]{Property}

\newtheorem{lemma}[remark]{Lemma}

\newtheorem{corollary}[remark]{Corollary}

\newtheorem{question}[remark]{Question}

\newtheorem{example}[remark]{Example}

\begin{document}
\title{Enlargeable Length-structures and Scalar Curvatures}
\author{Jialong Deng}
\date{}
\newcommand{\Addresses}{{
  \bigskip
  \footnotesize

  \textsc{Mathematisches Institut, Georg-August-Universit{\"a}t, G{\"o}ttingen, Germany}\par\nopagebreak
  \textit{E-mail address}: \texttt{jialong.deng@mathematik.uni-goettingen.de}}}
\maketitle
\begin{abstract}
  We define  enlargeable length-structures on closed topological manifolds and then  show that the connected sum of a closed $n$-manifold  with an enlargeable Riemannian length-structure  with an arbitrary closed smooth manifold carries no Riemannian metrics with positive scalar curvature.  We show that  closed smooth manifolds  with a locally CAT(0)-metric which is strongly equivalent to a Riemannian metric are examples of closed manifolds  with an enlargeable Riemannian length-structure. Moreover, the result is correct in arbitrary dimensions based on the main result of a recent paper by Schoen and Yau. 
  
   We define the positive  $MV$-scalar curvature  on  closed orientable topological manifolds and show the compactly enlargeable length-structures are the obstructions of its  existence.
\end{abstract}
\par 
\section{Introduction}
 \qquad The search for obstructions to the existence of a Riemannian metric with positive scalar curvature (PSC-metric) on a closed non-simply connected smooth manifold is an ongoing project. Based on index theory methods, Gromov and Lawson introduced enlargeability as an obstruction  based on index theory in \cite{MR569070}. Later, they relaxed the spin assumption in dimensions less than 8 in \cite[section~12]{MR720933}. Recently, Schoen and Yau showed that  the manifold  $N^n \# T^n$ carries no PSC-metrics  by geometric measure theory and minimal surfaces  methods in \cite{2017arXiv170405490S}, where $N^n$ is any closed oriented smooth $n$-manifold and $T^n$ is a torus.  Then, Cecchini and Schick  used Schoen and Yau's results  to show that a closed enlargeable manifold cannot carry any PSC-metric in \cite{MR4232210}. 
 
 Both enlargeabilities mentioned above are defined on Riemannian metrics and need at least $C^1$-smoothness for the maps. Here an enlargeable length-structure will be defined for   length metric spaces and  the maps only require to be continuous. Combining enlargeable length-structures with Gromov's Spherical Lipschitz Bound Theorem (SLB Theorem) \cite{MR3816521},  a new obstruction to the existence of PSC-metrics and positive  $MV$-scalar curvature $Sc^{MV}$ on a  closed manifold is given.  Details will be given later. 

\begin{theoremA}\label{A}
  Let  $X^n$  $(2\leq n \leq 8)$ be a closed orientable smooth manifold, then $X^n$ carries no PSC-metrics in its enlargeable length-structures. 
\end{theoremA}

\begin{theoremB}
Let $N^n$ $(2 \leq n \leq 8)$ be an arbitrary closed oriented smooth $n$-manifold and $M^n$ be a closed manifold with an enlargeable Riemannian length-structure. Then $N^n \# M^n$ does not admit a PSC-metric.
\end{theoremB}

\begin{remark}
In fact,  $N^n \# M^n$ does not admit a complete uniformly PSC-metric for any oriented manifold $N^n$ $(2 \leq n \leq 8)$.
\end{remark}

 \begin{theoremC}\label{C}
  If $(X^n, d, \mathcal{H}^n_d)$ $(n\leq 8)$ satisfies  $Sc^{MV}(X^n)\geq \kappa>0 $, then $d$ is not in any  compactly enlargeable  length-structures on the closed oriented topological manifold  $X$.
 \end{theoremC}

 Based on Schoen and Yau's argument \cite{2017arXiv170405490S}, Gromov proved the SLB Theorem \cite[Section~3]{MR3816521} for any dimensions. Similarly, another consequence of \cite{2017arXiv170405490S} is that  theorems A and B are also true in higher dimensions.

\begin{theoremD}\label{B'}
If $N^n$ is an arbitrary closed oriented smooth $n$-manifold and $X^n$ is  a closed manifold with an enlargeable Riemannian length-structure, then $N^n \# X^n$ carries no PSC-metrics. 

In particular, if $M^n$ is a closed manifold with an enlargeable Riemannian length-structure containing  a  locally CAT(0)-metric, then $N^n \# M^n$ does not admit a PSC-metric.
\end{theoremD}

\begin{remark}
Gromov also stated the result which said that a closed spin manifold with a locally CAT(0)-metric carries no PSC-metrics without proof in \cite[Section~4.1.2]{2019arXiv190810612G}.  Theorem B$^\prime$ supports Gromov's assertion.
\end{remark}

\begin{remark}
Though Theorem  \ref{B'} is similar to  Cecchini and Schick's Theorem A \cite{MR4232210} and the starting point of the proofs is  Schoen and Yau's new results \cite{2017arXiv170405490S}, the techniques of the proofs are different.  Cecchini and Schick's  proof  used the  standard constructions from geometric measure theory, especially, no symmetrization and no manifolds with boundary, whereas the proof of Theorem \ref{B'} is an application of Gromov's Spherical Lipschitz Bound Theorem that  uses these two ingredients.
\end{remark}

\begin{remark 2}
 Theorem  \ref{B'} is also inspired by an open conjecture that a closed aspherical manifold does not admit a PSC-metric.  Kasparov and Skandalis \cite{MR1998480} used KK-theory to  prove the strong Novikov conjecture for CAT(0)-group, implying that any locally CAT(0)-manifold carries no PSC-metrics. 
\end{remark 2}

  The fact that a closed Riemannian manifold with non-positive sectional curvature is enlargeable is well-known since Gromov and Lawson first proposed the definition of enlargeable manifolds \cite{MR569070}. Riemannian metrics with non-positive sectional curvature (non-positive curvature metric) are locally CAT(0)-metrics. If a manifold of dimension 2 or 3 admits a locally CAT(0)-metric, then it also admits non-positive curvature metrics,  according to the classic surface theory and the Thurston-Perelmann Geometrization Theorem  \cite[Proposition~1]{MR2872552}. 
 
 But there is a difference between locally CAT(0)-metrics and non-positive curvature metrics in dimensions$\geq$ 4. Aravinda and Farrell \cite{MR1312678} showed that the existence of non-positive curvature metric is not  a homeomorphism invariant in general, but the existence of a locally CAT(0)-metric is homeomorphism invariant. The existence of a non-positive curvature metric depends on that of the smooth structure. Furthermore, locally CAT(0)-manifolds which do not support a smooth structure in dimensions$\geq$ 5 were constructed by Davis and Hausmann \cite{MR1000373}.

  Aspherical $n$-manifolds ($n\geq 4$) that are not covered by Euclidean space  were first constructed  by Davis \cite{MR690848}.  Davis, Januszkiewicz and Lafont \cite{MR2872552} constructed a closed smooth four-dimensional manifold $M^4$,  which supports locally CAT(0)-metrics and whose universal cover $\widetilde{M^4}$ is diffeomorphic to $\mathbf{R^4}$, but $\pi_1$ is not isomorphic to the fundamental group of any compact Riemannian manifold with non-positive curvature. In dimensions$\geq$ 5, Davis and Januszkiewicz \cite[(5b.1)~ Theorem]{MR1131435} constructed a topological locally CAT(0)-manifold,  whose  universal cover is not homeomorphic to $\mathbf{R^n}$. They also constructed a smooth locally CAT(0)-manifold whose universal cover is homeomorphic to $\mathbf{R^n}$, but the boundary at infinity is distinct from $S^{n-1}$ in (5c.1) Theorem. Furthermore, those locally CAT(0)-metrics in (5c.1) Theorem  are strongly equivalent to the length metrics induced by Riemannian metrics since they come from  simplicial metrics for the smooth triangulations of the smooth manifold and the hyperbolization of polyhedra.
   
   More examples of those kind of exotic aspherical manifolds  can be found in \cite{MR1470741} and  \cite[Section~3]{MR2872552}.
 
  Therefore, the connect sum of a closed $n$-manifold and Davis's exotic aspherical $n$-manifold for $n\leq 8$, as  new examples  be detected by our methods, carries no PSC-metrics.

\begin{remark}
  Let $M^n$ $(n\geq 5)$ be the locally CAT(0)-manifold with center-free fundamental group, whose   universal cover is distinct from $\mathbf{R^n}$, as above-mentioned, and $N$ be an arbitrary closed locally CAT(0)-manifold with center-free fundamental group, then the product $M^n \times N$ is a locally CAT(0)-manifold, which does not support any non-positive curvature metrics. Otherwise, if  $M^n \times N$ carries a non-positive curvature metric, then, by Lawson-Yau's splitting theorem \cite{MR334083},  $M^n \times N$ is homeomorphic to  $M_1\times N_1$ such that $M_1$ and $N_1$ are endowed with non-positive curvature metrics and  $\pi_1(M_1)=\pi_1(M^n)$ and $\pi_1(N_1)=\pi_1(N)$. And  then we use the proof of Borel conjecture for Riemannian manifold with non-positive sectional curvature  by Farrell and Jones \cite{MR1216623} \cite{MR1159252} to show that $M_1$ is homeomorphic to $M^n$. Thus, $M_1$ also admits the locally CAT(0)-metric such that  the universal cover is distinct from $\mathbf{R^n}$. That is a contradiction. More details of the proof can be found in   \cite[Proposition~2]{MR2872552}.
\end{remark}

 Furthermore, using Davis's construction, Sapir \cite[Corollary~1.2]{MR3110794} firstly created closed aspherical topological $n$-manifolds $(n\geq 4)$ whose fundamental groups coarsely contain expanders and the aspherical $n$-manifold can be  smooth if $n\geq 5$.  Thus, Sapir's aspherical manifolds have infinite asymptotic dimension, are not coarsely embeddable into a Hilbert space, do not satisfy G. Yu's property A,  do not satisfy the Baum-Connes conjecture with coefficients.  Using Davis's construction and Sapir’s techniques, Osajda constructed  closed aspherical topological $n$-manifolds $(n\geq 4)$ whose fundamental groups contain coarsely embedded expanders in \cite[Corollary~3.5]{MR4176066}.  Therefore,  Sapir's and Osajda's aspherical $n$-manifolds $(n\leq 8)$ do not admit  PSC-metrics by our results.

 The paper is organized as follows: In Section \ref{Enlargeable}, we define    enlargeable length-structures and   prove Theorems  \ref{A} and \ref{B'}. In Section \ref{MV}, we define the positive  $MV$-scalar curvature  on  closed orientable topological manifolds and show the compactly enlargeable length-structures are the obstructions of its existence.

 $\mathbf{Acknowledgment}$: The author thanks Thomas Schick for the stimulating conversations, Simone Cecchini for his linguistic assistance during the preparation of this manuscript,  the referee's detailed comments, and  China Scholarship Council for its funding. 
 
 I learned F. Thomas Farrell and his coauthor's results from his inspirational courses in Yau Center (in Tsinghua University). The note is dedicated to Professor F. Thomas Farrell's 80th birthday.

\section{Enlargeable  length-structures, PSC-metrics and locally CAT(0)-metrics}\label{Enlargeable}

\qquad  A metric space $(X, d)$ is a length metric space if the distance between each pair of points equals the infimum of the lengths of curves joining the points.  A closed topological manifold admits length metrics \cite{MR52763}. Metrics induced by smooth Riemannian metrics on a  closed smooth manifold are length metrics and any two such length metrics are strongly equivalent, i.e. for any length metrics $d$ and $d'$, there exists a $c,C>0$ such that $cd\leq d'\leq Cd$. Moreover, length metrics can be induced by continuous Riemannian metrics on a closed smooth manifold \cite{MR3358543}.   However, unlike Riemannian metrics   on compact smooth manifolds, different length metrics may not be strongly equivalent to each other. For instance, two length metrics, among which one is induced by a Riemannian metric and the other  by a Finsher metric, are topologically equivalent, i.e. they induce the same topology,  but there may not be strongly equivalent in general.

 If $\pi: \hat{X}\to X$ is a covering map, then length metrics, being local, lift from $X$ to $\hat{X}$. Furthermore, every length metric $d$ on $X$ lifts to a unique length metric $\hat{d}$ for which the covering map is a local isometry.
 
 A topological manifold $X$ endowed with a complete length metric is called \textit{$\varepsilon$-hyperspherical }if it admits a continuous map $f$ to $S^n$ $(n=\mathrm{dim}(X))$ which is constant at infinity, of non-zero degree and  such that
 \begin{align*}
  \mathrm{Lip}(f):= \sup\limits_{\mathop{a \not= b}\limits_{{a,b}\in X}}\frac{d_{S^n}(f(a),f(b))}{d_X(a,b)} <\varepsilon.
 \end{align*}
  Here  constant at infinity means that there is a compact subset such that $f$ maps the complement of the compact subset to a point in $S^n$  and  $S^n$ is endowed with standard round metric. From now on, $S^n$ is also a length metric space  induced by the standard round metric $d_{S^n}$.

\begin{definition}[Enlargeable length-structures]
  A length metric $d$ on a closed orientable n-dimensional   topological manifold $X^n$ is said to be enlargeable if for each $\varepsilon > 0$ there is an oriented covering manifold $\widetilde{X^n}$ endowed with the induced metric $\widetilde{d}$ that is $\varepsilon$-hyperspherical.
  
  An enlargeable length-structure on $X^n$ is a strongly equivalent class of an enlargeable metric. 
  
  An enlargeable Riemannian length-structure on a closed orientable smooth manifold is  an enlargeable length-structure that contains   a length metric induced by a Riemannian metric on the manifold.
\end{definition}
 
The length metric induced by a Riemannian flat metric on $T^n$ is an enlargeable metric and $T^n$ endowed with this enlargeable length-structure is an  important example of the manifolds with an  enlargeable length-structure.   Enlargeable length-structures   can also be defined on a closed non-orientable manifold by lifting the metric onto the orientation cover.

\begin{remark}
  The enlargeable length-structure may be used to deal with positive scalar curvature in the metric geometry setting. For instance,   the  definition of scalar curvature for length metrics was given in \cite{MR3884658}.  
  
 Besides, the existence of a PSC-metric depends on the smooth structure.   Trying to use length structure to  study the PSC-metrics in Riemannian geometry, we define the enlargeable Riemannian length-structure.
\end{remark}

 \begin{question}
Is there a  closed topological manifold  with two different enlargeable length-structures? In particular, is there an enlargeable length-structure on closed smooth manifold such that the manifold does not carry an  enlargeable Riemannian length-structure?
\end{question}

\begin{remark}
  The following Lemma \ref{CAT}  shows that a locally CAT(0)-metric is enlargeable only on the assumption that it is strongly equivalent to a Riemannian metric. Especially, the example from the introduction of a non-smoothable manifold with a locally CAT(0)-metric is not known to be enlargeable.

\end{remark}
 
 \begin{property}[Properties of the enlargeable metric]\leavevmode \label{Enlargabel metric}
\begin{enumerate}
 \item[(1)]Let $(X,d_X)$ and $(Y,d_Y)$ be closed  orientable manifolds with length metrics and suppose there exists a Lipschitz continuous map $F:(X,d_X) \rightarrow (Y,d_Y)$ of non-zero degree. If  $d_Y$ is an enlargeable metric on $Y$, then so is $d_X$. 

\item[(2)]  The product of two enlargeable metrics is an  enlargeable metric. 
\end{enumerate}
\end{property}

\begin{proof}
Suppose that $\widetilde{Y}\to Y$ is a covering space and $\widetilde{Y}$  is $\varepsilon$-hyperspherical, i.e. it admits a continuous map $f:\widetilde{Y}\to S^n$ $(n=\mathrm{dim}(Y))$ that is constant at infinity, of non-zero degree and  such that $ \mathrm{Lip}(f)<\varepsilon$.  Let $\widetilde{X}\to X$ be the covering corresponding to the subgroup $F_*^{-1}(\pi_1(\widetilde{Y}))$ and $\widetilde{X}$ be endowed with the induced metric $\widetilde{d_X}$, then $F$ can be lifted to a proper map $\widetilde{F}:\widetilde{X}\to\widetilde{Y}$,  which is still a  continuous map with Lipschitz constant $\mathrm{Lip}(F)$. It implies the map $f\circ \widetilde{F}$ is constant at infinity, $\mathrm{Lip}(f\circ \widetilde{F})<\mathrm{Lip}(F)\varepsilon$  and of non-zero degree. Thus $d_X$ is an enlargeable metric on $X$.

 The fact that the composed map 
 \begin{align*}
 S^n\times S^m \rightarrow S^n \wedge S^m \rightarrow S^{n+m}
 \end{align*}
 
 is Lipschitz continuity and  non-zero degree implies (2).
\end{proof}

 The following two corollaries are immediate consequences of Property \ref{Enlargabel metric}.
 
\begin{corollary}
If a closed manifold $X$ carries an  enlargeable length-structure, then $X\times S^1$ still carries an enlargeable length-structure.
\end{corollary}

\begin{corollary}\label{Enlargeable Riemannn}
Let $(X,d_X)$  be a closed smooth manifold, where $d_X$  is induced by a Riemannnian metric, and $(Y,d_Y)$ be another  manifold, where $d_Y$ is in the enlargeable length-structure. Suppose there exists a Lipschitz continuous map $F:(X,d_X) \rightarrow (Y,d_Y)$ of non-zero degree, then $X$ carries an enlargeable Riemannian  length-structure. 
\end{corollary}

\begin{theorem}[SLB Theroem  \cite{MR3816521}]
If the scalar curvature of a (possibly incomplete) Riemannian n-manifold $X^n$ $(2\leq n \leq 8)$ is bounded from below by $n(n-1)$, then for all continuous maps $f$ from $X^n$ to the sphere $S^n$ that are constant at infinity and of non-zero degree,  they hold that $\mathrm{Lip}(f)>\frac{C}{\sqrt{n}\pi}$.
Here $S^n$ is endowed with the standard round metric and  $C>\frac{1}{3}$.
\end{theorem}

\begin{question}\label{Lipschitz}
  Let $(N^n, g)$ be a complete Riemannian $n$-manifold (compact or non-compact) with scalar curvature bounded below by $n(n-1)$. Let  $f$ be a continuous map from $N^n$ to the sphere $S^n$ with standard round metric of non-zero degree that is constant at infinity. Is  $\mathrm{Lip}(f)$ bounded from blow by one?
\end{question}

\begin{remark} 
  A positive answer to Question \ref{Lipschitz} would allow us to also cover Llarull's rigidity theorems \cite[Theorem~A~and~B]{MR1600027},  Lohkamp's results \cite{2018arXiv181211839L}, and their remarkable corollaries.
\end{remark}

  Gromov aimed to improve the lower bound of the Lipschitz constant in \cite[Section~3]{MR3816521} and discussed that the Lipschitz bound would be 1. Here,  the existence of a uniformly positive  lower bound of the Lipschitz constant is used as the main tool to detect the  obstruction to the  existence of  PSC-metrics. 

\begin{proof}[Proof of Theorem \ref{A}]
   Assume $X^n$ carries an enlargeable length-structure, there exists a length metric $d$ such that one orientable covering $(\widetilde{X^n},\tilde{d})$ is $\varepsilon$-hyperspherical ($\varepsilon$ small enough). If $X^n$  admits a PSC-metric $g$ in the enlargeable length-structure, then the Lipschitz constant of all maps (maps are constant at infinity and non-zero degree)  from $(\widetilde{X^n},\tilde{g})$ to $S^n$ has a uniformly positive lower bound $C$ by the SLB Theorem. Besides, there are positive constants $\alpha_1$ and $\alpha_2$ such that $\alpha_1d\leq d_g \leq \alpha_2d $  by  the definition of enlargeable length-structure, where $d_g$ is the induced metric by $g$ on $X^n$. Then the Lipschitz constant of the map from $(\widetilde{X^n},\tilde{d})$ to $S^n$ has the uniformly lower bound $\alpha_1C$, which contradicts  the $\varepsilon$-hypersphericity.
\end{proof}

  Let $(X, d_X)$ be a length space. A geodesic triangle $\bigtriangleup$ in $X$ with geodesic segments as its sides is said to satisfy the CAT(0)-inequality if it is slimmer than the comparison triangle in the Euclidean plane, i.e. if there is a comparison triangle $\bigtriangleup'$ in the Euclidean plane with sides of the same length as the sides of $\bigtriangleup$ such that the distance between points on $\bigtriangleup$ is less than or equal to the distance between corresponding points on $\bigtriangleup'$. A length metric $d$ on $X$ is said to be a locally CAT(0)-metric if every point in $X$ has a geodesically convex neighborhood, in which every geodesic triangle satisfies the CAT(0)-inequality.
  
 A locally CAT(0)-manifold is a  topological manifold endowed with a locally CAT(0)-metric. Gromov generalized the classic Hadamard-Cartan theorem to locally CAT(0)-manifolds  \cite{MR823981}: the universal cover of a locally CAT(0)-manifold endowed with the induced metric is a globally CAT(0)-manifold, i.e. every two points can be connected by a unique geodesic line and every geodesic  triangle on it satisfies the CAT(0)-inequality. Thus a locally CAT(0)-manifold is aspherical, i.e. its universal cover is contractible. 
 
\begin{lemma}\label{CAT}
A locally CAT(0)-metric which is strongly equivalent to a Riemannian metric on a closed smooth manifold is an enlargeable metric.
\end{lemma} 

\begin{proof}
 Let $(M^n,d)$ be the  closed $n$-dimensional smooth locally CAT(0)-manifold, then its universal cover $(\widetilde{M^n},\tilde{d})$ is a globally CAT(0)-manifold by Gromov's Theorem. Consider the map
\begin{align*}
f_t: \widetilde{M^n}\rightarrow \widetilde{M^n}  \quad  x \rightarrow  \gamma_x(t\tilde{d}(x,x_0)),
\end{align*}
where $x_0$ is a fixed point in $\widetilde{M^n}$, $t\in (0,1]$ and $\gamma_x$ is the unique geodesic segment from $x$ to $x_0$. It is well-defined by the property of globally CAT(0)-manifolds and is a proper map such that the degree of $f_t$ is non-zero. By the CAT(0)-inequality applied to the geodesic triangle with endpoints $x$, $y$ and $x_0$, one gets 
\begin{align*}
\tilde{d}(f_t(x),f_t(y))\leq td_\mathbf{R^2}(x,y)=td(x,y)
\end{align*}
 for $x,y \in \widetilde{M^n}$. Therefore, $\mathrm{Lip}(f_t)\leq t$.

 Let $\pi: (\widetilde{M^n},\tilde{d}) \rightarrow S^n$ be a collapse map around $x_0$. Then the degree of $\pi$ is 1  and $\mathrm{Lip}(\pi)\leq C$ by the smoothness of $\pi$, the fact that $d$ is strongly equivalent to a Riemannian metric and the compactness of $M^n$. Thus,  $\pi \circ f_t:\widetilde{M^n} \rightarrow S^n$ has non-zero degree and $\mathrm{Lip}(\pi \circ f_t) \leq tC$. For any small $\varepsilon > 0$, we can choose $t$ such that $(\widetilde{M^n},\tilde{d})$ is $\varepsilon$-hyperspherical. That means that a  manifold  endowed with a locally CAT(0)-metric is enlargeable.
\end{proof} 

\begin{remark}
The assumption that the locally CAT(0)-metric is strongly equivalent to a Riemannian metric is used in the argument of  $\mathrm{Lip}(\pi)\leq C$. It is not clear if the condition can be dropped for Lemma \ref{CAT}. 
\end{remark}

\begin{proof}[Proof of Theorem \ref{B'}]
 Combining  Lemma \ref{CAT}, Corollary \ref{Enlargeable Riemannn},  and  the fact of the strong equivalence of all Riemannian metrics on a closed smooth manifold, we conclude Theorem  \ref{B'}.
\end{proof}

\section{$Sc^{MV}\geq\kappa$}\label{MV}
\qquad To generalize the notion of PSC-metrics to non-Riemannian metric space, for example to piecewise smooth polyhedral spaces,  Gromov  \cite[Section~5.3.1]{2019arXiv190810612G} defined the max-scalar curvature $Sc^{max}$  as follows:
\begin{definition}[Gromov]
Given a metric space $P$ which is locally compact and locally contractible, and a homology class $h\in H_n(P)$ defines $Sc^{max}(h)$ as the supremum of the numbers $\sigma\geq 0$, such that there exists a closed orientable Riemannian $n$-manifold $X$ with $Sc(X)\geq\sigma$ and a 1-Lipschitz map $f: X\to P$, such that the fundamental homology class $[X]$ goes to $h$,  $f_*[X]=h$. Here $Sc(X)$ is the scalar curvature of $X$ and 1-Lipschitz map $f$ means $\mathrm{Lip}(f)\leq 1$.
\end{definition}

\begin{remark}
The definition makes sense  without assuming $\sigma\geq 0$. But if an $h\in H_n(P)$ $(n\geq 3)$ can be represented by the fundamental homology class $[X]$, then we always have $Sc^{max}(h)\geq 0$. Since  a closed orientable smooth  $n$-manifold $(n\geq 3)$ admits Riemnanian metrics with constant negative scalar curvature by Kazdan-Warner theorem \cite{MR365409}, then one can scale the metric such that the Lipschitz constant is small and the scalar curvature is closed to $0$.

\end{remark}

 Though $Sc^{max}([X])\geq \inf\limits_xSc(X,x)$ for all  closed Riemannian  manifold $X$ as observed by Gromov, the positivity of the max-scalar curvature cannot imply that it carries a PSC-metric in general. For instance, let $\Sigma$ be  the  exotic sphere which does not admit PSC-metrics, whose existence  was showed by Hitchin in \cite{MR358873}, and $g_{\Sigma}$ be a Riemanniam metric on it. Then one can scale the round metric on the standard sphere such that the Lipschitz constant of the identity map is smaller than 1. Thus one has $Sc^{max}([\Sigma])> 0$, but $\Sigma$  does not admit PSC-metrics.  
    
Furthermore, let $(N,g_N)$ be a closed oriented Riemannian manifold with a PSC-metric $g_N$ and   $M$ be a closed oriented smooth manifold. Assume that there exists a degree one smooth map $f:N\to M$,  then one gets   $Sc^{max}([M])\geq 0$ with an arbitrary smooth Riemannian metric $g$ on $M$. That means $Sc^{max}([M])\geq 0$, even when the scalar curvature of $(M,g)$ is $-1$.

    Therefore, the definition of the max-scalar curvature on a metric space needs to be improved.  Gromov also proposed  the $n$-volumic scalar curvature on metric measure spaces in \cite[Section~26]{Gromov}. The $n$-volumic scalar curvature  may be too general to be useful, for more discussion see  \cite{MR4210894}.     The following definition is trying to mix the  max-scalar curvature and $n$-volumic scalar curvature on  metric measure spaces.  

Let $(S^2(R),d_S, vol_S)$ be a Riemannian 2-sphere  endowed with the round metric such that the  scalar curvature equals to $2R^{-2}$ and that $(\mathbf{R}^{n-2},d_E, vol_E)$ is endowed with Euclidean metric, then the product manifold  $S^2(R)\times \mathbf{R}^{n-2}$ can be endowed with the Pythagorean product metrics $d_{S\times E}:=\sqrt{d_S^2 + d_E^2}$   and the volume $vol_{S\times E}:=vol_S \otimes vol_E$.

From now on, let $X^n$ be a closed topological $n$-manifold and $d$ be the length metric such that the Hausdorff dimension of $(X^n, d)$  is $n$, i.e $\mathrm{dim}_H(X^n)=n$. Therefore,  there exists the $n$-dimension Hausdorff measure $\mu_n$ on it and then we normalize it by  $\mathcal{H}^n_d=\mathcal{W}_n\mu_n$, where $\mathcal{W}_n$ is the  $n$-dimensional volume (the Lebesgue measure) of a Euclidean ball of radius 1 in the $n$-dimensional Euclidean space,  so that  $\mathcal{H}^n_{d_E}=vol_E$ for  $\mathbf{R}^{n}$ with the Euclidean metric $d_E$. Therefore, we have the   metric measure space  $(X^n,d,  \mathcal{H}^n_d)$ and will focus on this kind of metric measure space $(X^n,d,  \mathcal{H}^n_d)$ in this paper.

\begin{example}[Examples of metric measure spaces  $(X^n,d,  \mathcal{H}^n_d)$]~\\
\begin{enumerate}
\item[1.]  Smooth oriented Riemannian manifolds $(M^n,g)$  with the induced volume forms $(M^n,d_g,vol_g)$  are the fundamental examples. 
\item[2.]  A length metric that is strongly equivalent to a Riemannian metric satisfies the requirement, since the Hausdorff dimension is bi-Lipschitz invariant. 
\item[3.]  Locally CAT(0)-manifolds with induced Hausdorff measure are examples, as the Hausdorff dimension of a closed topological $n$-manifold with a locally CAT(0)-metric is $n$.
\item[4.]  If $X^n$  admits an length metric $d'$ such that $(X^n,d')$ is a Alexandrov space with curvature bounded from below, then  $\mathrm{dim}_H(X^n)=n$.  Then the Alexandrov space is also an example. 
\end{enumerate}
\end{example}

Note that there is only one reasonable notion of volume for Riemannian manifolds. But one can define Finslerian volumes for  Finsler metrics in different ways and obtain essentially different results  \cite[Proposition~5.5.12]{MR1835418}.

 We will define the positive of $MV$-scalar curvature $Sc^{MV}$ on the metric measure space  $(X^n,d,  \mathcal{H}^n_d)$. 

\begin{definition}[$Sc^{MV}\geq\kappa$]
  The $MV$-scalar curvature of $X^n$ is bounded from below by $\kappa>0$ for  $X^n=(X^n, d,  \mathcal{H}^n_d)$, i.e. $Sc^{MV}(X^n)\geq\kappa>0$,  if the closed oriented topological $n$-dimensional  $X^n$ satisfies the following two conditions:
\begin{itemize}
\item[(1).] The metric space $(X^n,d)$   satisfies  $Sc^{max}([X^n])\geq \kappa$ for the fundamental class  $[X^n]\in H_n(X^n;\mathbf{Z})$. 

\item[(2).] The metric measure space $X^n$ is locally volume-wise smaller than  $S^2(R)\times \mathbf{R}^{n-2}=(S^2(R)\times \mathbf{R}^{n-2},d_{S\times E},vol_{S\times E})$ for all $R>\frac{\sqrt{2}}{\sqrt{\kappa}}$, i.e. for  $R>\frac{\sqrt{2}}{\sqrt{\kappa}}$,  there is an $\epsilon_R$, which depends on $R$, such that   all $\epsilon_R$-balls in $X$ are smaller than the $\epsilon_R$-balls in $S^2(R)\times \mathbf{R}^{n-2}$, $\mathcal{H}^n_d (B_x(\epsilon_R))<vol_{S\times E}(B_{x'}(\epsilon_R))$, for all $x\in X$ and $ x'\in S^2(R)\times \mathbf{R}^{n-2} $. 
\end{itemize}
\end{definition}

Since $Sc^{max}([X^n])=Sc^{max}(-[X^n])$, the definition is independent from the chosen of   the orientation. And  $Sc^{MV}(X^n)\geq\kappa$ is invariant under the isomorphic transformation. Here an isomorphic transformation means there is $f: (X,d,\mu)\to (X',d',\mu')$ such that $f_*\mu'=\mu$ and $f$ is isometric between $d$ and $d'$.  Thus, the definition  is well-defined.

 \begin{proposition}
 Let $g$ be a $C^2$-smooth Riemannian metric on a closed oriented $n$-manifold $M^n$ with induced metric measure space $(M^n, d_g, vol_g)$, then the scalar curvature of $g$ is bounded from below by $\kappa>0$ if and only if $Sc^{MV}(M^n)\geq \kappa$.
\end{proposition}  

\begin{proof}
  Assume  the scalar curvature of $g$ is bounded from below by $\kappa>0$,  then the volume formula of a small ball, 
  \begin{align*}
  vol_g(B_x(r)) = vol_E(B_r)\left[1 - \frac{Sc_g(x)}{6(n+2)} r^2 + O(r^4)\right]
  \end{align*}
   as $r \rightarrow 0$, where $B_x(r)$ is an $r$-ball in $M^n$ and  $B_r$ is an $r$-ball in $\mathbf{R}^n$, implying condition (2) in the definition of  $Sc^{MV}(M^n)\geq \kappa$. And $Sc^{max}([M^n])\geq \inf\limits_xSc(M^n,x)\geq \kappa$ implies condition (1). 
   
   On the other hand, if  $Sc^{MV}(M^n)\geq \kappa$,  then $Sc_g \geq \kappa > 0$.  Otherwise, assume there exist small $\epsilon>0$ such that  $Sc_g \geq \kappa-\epsilon>0$. That means  there exists a point $x_0$ in $M^n$ such that $Sc_g(x_0)= \kappa-\epsilon$, as $M^n$ is compact and the scalar curvature is a continuous function on $M^n$. Thus, we can find a small $r$-ball $B_r(x_0)$ such that the volume of   $B_r(x_0)$ is greater than the volume of the $r$-ball in the $S^2(\gamma)\times \mathbf{R}^{n-2}$ for $\gamma=\sqrt{\frac{2}{\kappa-\frac{\epsilon}{2}}}$, which is a contradiction. 
\end{proof}  

\begin{remark}
The existence of length metrics on $X^n$ with  $Sc^{MV}(X^n)\geq \kappa$ is  the  invariant under homeomorphisms. However, the positivity of the $Sc^{MV}$ cannot imply that it carries a PSC-metric in general.

 For instance, one can use the identity map between the  exotic sphere $\Sigma^n$  above  and the standard sphere to pull back the length metric which is induced by the standard round metric to $\Sigma^n$, i.e. giving $id: \Sigma^n\to (S^n, d_{S^n})$ gets $(\Sigma^n, id^*d_{S^n})$. Then one has $Sc^{MV}(\Sigma^n)\geq n(n-1)$ for the metric measure space $(\Sigma^n, id^*d_{S^n}, \mathcal{H}^n_{id^*d_{S^n}})$. But $id^*d_{S^n}$ is not induced by any  $C^2$-smooth Riemannian metric on $\Sigma^n$.
\end{remark}

\begin{remark}
The example of the exotic sphere above shows that the condition (1) cannot imply the condition (2) in the definition of $Sc^{MV}\geq\kappa$ in general.  The condition (2) also cannot imply the condition (1) in general.  Since one can choose the length metric induced by a Finsler metric such that the induced Hausdorff measure is  locally volume-wise smaller than  $S^2(R)\times \mathbf{R}^{n-2}=(S^2(R)\times \mathbf{R}^{n-2},d_{S\times E},vol_{S\times E})$ for all $R>\frac{\sqrt{2}}{\sqrt{\kappa}}$,  and the Finsler metric is not bi-Lipshchitz equivalent to a Riemannian metric in general.
\end{remark}

\begin{question}\label{Aspherical}
Let $N^n$ be a closed orientable aspherical $n$-manifold. Does there exist an orientable closed Riemannian $n$-manifold $M^n$ such that there exists a non-zero degree map $f$ from  $M^n$ to $N^n$?
\end{question}

\begin{remark}
 Question \ref{Aspherical} is inspired by the conjecture  that a closed aspherical manifold does not carry  PSC-metrics.  It is natural to ask that whether a closed aspherical manifold admits a length metric with  positive max-scalar curvature or positive of $Sc^{MV}$.
\end{remark}

\begin{proposition}[Quadratic Scaling]
 Assume that $(X^n, d,   \mathcal{H}^n_d)$ satisfies   $Sc^{MV}(X^n)\geq \kappa>0$, then $Sc^{MV}(\lambda X^n)\geq \lambda ^{-2}\kappa$ for all   $\lambda>0$, where $\lambda X^n:=(X^n, \lambda \cdot d,\lambda^n \cdot   \mathcal{H}^n_d)$.
\end{proposition}

\begin{proof}
If we scale $d$ by $\lambda\neq 0$, then  $\mathcal{H}^n_{\lambda d}=\lambda^n \mathcal{H}^n_d$. Combining it with the fact that the $Sc(\lambda^2g)=\lambda^{-2}Sc(g)$ for a smooth Riemannian metric will complete the proof.
\end{proof}

Let $\pi: \hat{X}^n\to X^n$ be a  covering map, then the length metric $d$ on $X^n$ is lifted to a unique length metric $\hat{d}$ such that the covering map is a local isometry.  Hence $dim_H(X,d)=dim_H(\hat{X},\hat{d})$ for a finite connected cover of $X$. Then we endow  the lifting length metric on a finite cover of $(X^n, d, \mathcal{H}^n_d)$ such that  $(\hat{X}^n,\hat{d}, \mathcal{H}^n_{\hat{d}})$ is a metric measure space.

\begin{proposition}\label{cover}
Assume  $(X^n, d,   \mathcal{H}^n_d)$ satisfies   $Sc^{MV}(X^n)\geq \kappa>0$ and
 $\hat{X}^n$   is a finite connected cover of $X^n$,    then    $Sc^{MV}(\hat{X}^n)\geq \kappa$ for $(\hat{X}^n,\hat{d}, \mathcal{H}^n_{\hat{d}})$.
\end{proposition}

\begin{proof}
As $(\hat{X}^n,\hat{d}, \mathcal{H}^n_{\hat{d}})$ is locally isometric to $(X^n, d,   \mathcal{H}^n_d)$ and  $Sc^{MV}(X^n)\geq \kappa>0$, $(\hat{X}^n,\hat{d}, \mathcal{H}^n_{\hat{d}})$ is also  locally volume-wise smaller than  $S^2(R)\times \mathbf{R}^{n-2}$.

Let $(M^n,g)$ be the closed orientable Riemannian manifold with $Sc(g)\geq \kappa$ such that $f:M^n\to X^n$ is $1$-Lipschitz map and $f_*([M^n])=[X^n]$. Then $f^*\hat{X}^n$ is a finite cover of $M^n$ and we denote it by $\hat{M}^n$, i.e.  $\hat{M}^n:=f^*\hat{X}^n$.  Then the Lipshicht constant of  $\hat{f}: (\hat{M}^n, \hat{g})\to (\hat{X}^n, \hat{d})$ is 1, where $\hat{g}$ is the lifting metric of $g$. Then we have the following two commutative diagrams.

 \begin{equation*}
 \xymatrix{\hat{M}^n \ar@{-->}[d]_{\pi'} \ar@{-->}[r]^{\hat{f}} & \hat{X}^n\ar[d]^{\pi} \\                                       
          M^n\ar[r]_f & X^n }                  \quad\quad        \xymatrix{H_n({\hat{M}^n};\mathbf{Z}) \ar@{-->}[d]_{\pi'_*} \ar@{-->}[r]^{\hat{f}_*} & H_n({\hat{X}^n};\mathbf{Z})\ar[d]^{\pi_*} \\                                       
          H_n({M^n};\mathbf{Z})\ar[r]_{f_*} & H_n({X^n};\mathbf{Z}) }                            
        \end{equation*}

    Using the wrong-way map, we can map  $[X^n]$ to $ H_n(\hat{X}^n;\mathbf{Z})$, denoting it by $[\hat{X}^n]$ and then we choose $[\hat{X}^n]$ as the fundamental class  of $\hat{X}^n$.   Again, we map $[\hat{X}^n]$ to $ H_n(\hat{M}^n;\mathbf{Z})$, denoting it by $[\hat{M}^n]$ and  choose $[\hat{M}^n]$ as the fundamental class  of $\hat{M}^n$, i.e.  $\hat{f}_*([\hat{M}^n])=[\hat{X}^n]$.
    
     Thus,  $Sc^{MV}(\hat{X}^n)\geq \kappa$. 
\end{proof}

\begin{remark}
One can also define $Sc^{MV}\geq\kappa>0$ for closed non-orientable  topological manifolds by requiring that the double cover with the induced metric satisfies   $Sc^{MV}\geq\kappa>0$, since the Hausdorff dimensions are equal to each other in this case.
\end{remark}
 
\begin{proposition}[ Weak SLB Theorem]\label{Weak SLB}
 Assume $(X^n, d,  \mathcal{H}^n_d)$ $(n\leq 8)$ satisfies  $Sc^{MV}(X^n)\geq \kappa>0$,  then for all continuous non-zero degree maps $h$ from $X^n$ to the sphere $S^n$,  it holds that $\mathrm{Lip}(h)>\frac{C\sqrt{n-1}}{\sqrt{\kappa}\pi}$.
Here $S^n$ is endowed with the standard round metric $d_{S^n}$ and  $C>\frac{1}{3}$.

\end{proposition}

\begin{proof}

 Given a small $\epsilon>0$,  there exists a closed orientable Riemannian $n$-manifold $(M^n, g)$ with $Sc(g)\geq \kappa-\epsilon>0$ and a 1-Lipschitz and degree one map $f:(M^n,d_g)\to (X^n,d)$  by the definition of max-scalar curvature. Let $h$ be a continuous non-zero degree map $h:(X^n,d)\to (S^n, d_{S^n})$, then
 \begin{align*}
 \mathrm{Lip}(h\circ f)\leq \mathrm{Lip}(h)\times\mathrm{Lip}(f)
\end{align*}  
  induces $\mathrm{Lip}(h)\geq \mathrm{Lip}(h\circ f)$.
 
If one scalars the metrics $d_g$ and $d$ by  a constant $\lambda\neq 0$,   then one has  $Sc(\lambda^2g)=\lambda^{-2}Sc(g)$, $d_{\lambda^2g}=\lambda d_g$. The new maps are denoted by $\tilde{h}: (X^n,\lambda d)\to (S^n,d_{S^n}) $ and $\tilde{f}:(M^n, \lambda d_g)\to (X^n,d)$, then one has 
 $ \mathrm{Lip}(\tilde{f})=\mathrm{Lip}(f)$, $\mathrm{Lip}(\tilde{h}\circ \tilde{f})=\lambda^{-1} \mathrm{Lip}(h\circ f)$ and $ \mathrm{Lip}(\tilde{h})=\lambda^{-1}\mathrm{Lip}(h)$.

 Choose $\lambda=\sqrt{\frac{n(n-1)}{\kappa-\epsilon}}$ such that $Sc(\lambda^2g)\geq n(n-1)$, then SLB Theorem implies that $\mathrm{Lip}(\tilde{h}\circ \tilde{f})>\frac{C}{\sqrt{n}\pi}$. That means 
 \begin{align*}
 \lambda^{-1}\mathrm{Lip}(h)=\mathrm{Lip}(\tilde{h})\geq \mathrm{Lip}(\tilde{h}\circ \tilde{f})>\frac{C}{\sqrt{n}\pi}.
 \end{align*}
 
Thus, we have $\mathrm{Lip}(h)>\frac{C\sqrt{n-1}}{\sqrt{\kappa-\epsilon}\pi}$. Let $\epsilon$ go to 0, then one has $\mathrm{Lip}(h)>\frac{C\sqrt{n-1}}{\sqrt{\kappa}\pi}$.

\end{proof}

Gromov and Lawson define the enlargeability  by allowing only finite coverings in \cite{MR569070} 
Hanke and Schick   \cite{MR2259056}  showed that the Rosenberg index of this kind of  enlargeable  spin manifold does not vanish,  which also implies that the manifold carries no PSC-metrics.

\begin{definition}[Compactly enlargeable length-structures]\label{Compactly}
  A length metric $d$ on a closed orientable n-dimensional   topological manifold $X^n$ is said to be compactly enlargeable if for each $\varepsilon > 0$ there is a finite connected covering manifold $\hat{X}^n$ endowed with the induced metric $\hat{d}$ which is $\varepsilon$-hyperspherical. (This notion is defined at the beginninig of Section 2).
  
  A compactly enlargeable length-structure on $X^n$ is the strongly equivalent class of an enlargeable metric. 
  
  A compactly enlargeable Riemannian length-structure on a closed orientable smooth manifold is  an enlargeable length-structure, which contains   a length metric induced by a Riemannian metric on the manifold.
\end{definition}

Let us recall  Theorem \ref{C} and give a proof yo it.

 \begin{theoremC}
  If $(X^n, d, \mathcal{H}^n_d)$ $(n\leq 8)$ satisfies  $Sc^{MV}(X^n)\geq \kappa>0 $, then $d$ is not in any  compactly enlargeable  length-structures on the closed oriented topological manifold  $X$.
 \end{theoremC}

\begin{proof}
 The argument is the same in the proof of Theorem \ref{A} .   Assume $d$ is in a compactly enlargeable length-structure, then  there exists a compactly enlargeable metric $d'$ on $X^n$ such that $\alpha_1d\leq d'\leq \alpha_2d$ for some $0<\alpha_1\leq \alpha_2$.   There exists a finite connected covering manifold $\hat{X}^n$ such that the induced metric $\hat{d'})$ is $\varepsilon$-hyperspherical for $ \varepsilon<\frac{C\sqrt{n-1}}{\sqrt{\kappa}\alpha_2\pi}$ by the definition of compactly enlargeable metric. That means it exists the non-zero degree map $h: (\hat{X},\hat{d'})\to S^n$ such that $\mathrm{Lip}(h)_{\hat{d'}}< \frac{C\sqrt{n-1}}{\sqrt{\kappa}\alpha_2\pi}$, where $\mathrm{Lip}(h)_{\hat{d'}}$ is the Lipschitz constant with respect the metric $\hat{d'}$. 
 
 On the other hand, we have 
 \begin{align*}
 \alpha^{-1}_2\mathrm{Lip}(h)_{\hat{d}}\leq \mathrm{Lip}(h)_{\hat{d'}} \leq \alpha^{-1}_1\mathrm{Lip}(h)_{\hat{d}}.
 \end{align*}
 But Proposition \ref{cover} shows that $Sc^{MV}(\hat{X}^n, \hat{d})\geq \kappa$ and then   Weak SLB Theorem \ref{Weak SLB}
shows that $\mathrm{Lip}(h)_{\hat{d}}>\frac{C\sqrt{n-1}}{\sqrt{\kappa}\pi}$, which is a contradiction. 

One can also prove the proposition by Corollary \ref{Enlargeable Riemannn} and the definition of max-scalar curvature.  If $d$ is in a compactly enlargeable length-structure and $f:(M^n,g)\to (X^n,d)$ is a  degree 1 map and $\mathrm{Lip}(f)\leq 1$, then $g$ is in an enlargeable Riemannian length-metric structure by Corollary \ref{Enlargeable Riemannn}. That means $M^n$ carries no PSC-metrics. Thus, $Sc^{max}([X^n])=0$, which is a contradiction.
\end{proof}

  The  scalar curvature of Riemannian metrics are additivity under Pythagorean-Riemannian products, however, $Sc^{MV}$ may not be additivity under Pythagorean products in general. Note that for the  Pythagorean product of two metric spaces $(X,d)$ and $(Y,d_1)$, $(X\times Y, d_2)$ with $d_2:=\sqrt{d^2+d^2_1}$ ,  we have 
\begin{align*}
 \mathrm{dim}_H(X)+\mathrm{dim}_H(Y)\leq \mathrm{dim}_H(X\times Y)\leq \mathrm{dim}_H(X)+\mathrm{dim}_B(Y),
\end{align*}
  where $\mathrm{dim}_B(Y)$ is the upper box counting dimension of $Y$,  and the inequality may be strict.  If $Y$ is a smooth Riemannian manifold, then $dim_H(Y)=dim_B(Y)$.  Furthermore, we  have 
  \begin{align*}
   \mathcal{H}^{n+m}_{d_2}\geq C(n,m)\mathcal{H}^n_d\times \mathcal{H}^m_{d_1}
  \end{align*}
  where $C(n,m)$ is a constant dependent only on $n$ and $m$ and  $C(n,m)\geq 1$.  $C(n,m)$ may be greater than 1, but  if $X$ and $Y$ are rectifiable Borel subsets of Euclidean space, then   $C(n,m)=1$ was showed by Federer in \cite[3.2.23~Theorem]{MR0257325}.

\begin{proposition}
 Assume $(X^n\times Y^m, d_2)$ is the  Pythagorean product of $(X^n, d,   \mathcal{H}^n_{d})$ and $(Y^m, d_1,  \mathcal{H}^n_{d_1})$, where $d_2:=\sqrt{d^2+d_1^2}$,   satisfies that  $\mathrm{dim}_H(X^n\times Y^m)=n+m$ and  the measure $  \mathcal{H}^{n+m}_{d_2}:=   \mathcal{H}^{n}_{d} \otimes  \mathcal{H}^{m}_{d_1}$. Then if $Sc^{MV}(X^n)\geq \kappa_1>0$ and $Sc^{MV}(Y^m)\geq \kappa_2>0$,   then $Sc^{MV}(X^n\times Y^m)\geq \kappa_1+\kappa_2$. 

\end{proposition}
 
 \begin{proof}
  Since $(X^n\times Y^m, d_2)$  is locally  volume-wise smaller than $(S^2(R_1)\times \mathbf{R}^{n-2})\times (S^2(R_2)\times \mathbf{R}^{m-2}) $ for $R_1>\frac{\sqrt{2}}{\sqrt{\kappa_1}}$ and $R_2>\frac{\sqrt{2}}{\sqrt{\kappa_2}}$,  since $ \mathcal{H}^{n+m}_{d_2}:=   \mathcal{H}^{n}_{d} \otimes   \mathcal{H}^{m}_{d_1}$. And  $(S^2(R_1)\times \mathbf{R}^{n-2})\times (S^2(R_2)\times \mathbf{R}^{m-2}) $ is locally  volume-wise smaller than $S^2(R_1+R_2)\times \mathbf{R}^{n+m-2}$ for $R_1+R_2>\frac{\sqrt{2}}{\sqrt{\kappa_1+\kappa_2}}$. Thus,  we have $(X^n\times Y^m, d_2)$  is locally  volume-wise smaller than $S^2(R_1+R_2)\times \mathbf{R}^{n+m-2}$ for $R_1+R_2>\frac{\sqrt{2}}{\sqrt{\kappa_1+\kappa_2}}$.
And one has
 \begin{align*}
 Sc^{max}([X^n]\otimes[Y^m])\geq Sc^{max}([X^n])+  Sc^{max}([Y^m])
 \end{align*}
 
 Hence $Sc^{max}([X^n]\otimes[M^m])\geq \kappa_1+\kappa_2$. 
 \end{proof}

\begin{question}
Assume that $d$ is a length metric on the  closed topological $n$-manifold $X^n$  such that  $(X^n,d)$ is an Alexandrov space with curvature$\geq \kappa >0$. Do we have $Sc^{MV}([X^n])\geq n(n-1)\kappa$?
\end{question}
Note that an Alexandrov space with curvature$\geq \kappa >0$ satisfies the Bishop-inequality \cite[Theorem~10.6.8]{MR1835418} and then it implies the condition (2) of the definition of  $Sc^{MV}\geq\kappa$.

\bibliographystyle{alpha}
\bibliography{reference}

\begin{thebibliography}{{Loh}18}

\bibitem[ADG97]{MR1470741}
F.~D. Ancel, M.~W. Davis, and C.~R. Guilbault.
\newblock {${\rm CAT}(0)$} reflection manifolds.
\newblock In {\em Geometric topology ({A}thens, {GA}, 1993)}, volume~2 of {\em
  AMS/IP Stud. Adv. Math.}, pages 441--445. Amer. Math. Soc., Providence, RI,
  1997.

\bibitem[AF94]{MR1312678}
C.~S. Aravinda and F.~T. Farrell.
\newblock Rank {$1$} aspherical manifolds which do not support any
  nonpositively curved metric.
\newblock {\em Comm. Anal. Geom.}, 2(1):65--78, 1994.

\bibitem[BBI01]{MR1835418}
Dmitri Burago, Yuri Burago, and Sergei Ivanov.
\newblock {\em A course in metric geometry}, volume~33 of {\em Graduate Studies
  in Mathematics}.
\newblock American Mathematical Society, Providence, RI, 2001.

\bibitem[BGS85]{MR823981}
Werner Ballmann, Mikhael Gromov, and Viktor Schroeder.
\newblock {\em Manifolds of nonpositive curvature}, volume~61 of {\em Progress
  in Mathematics}.
\newblock Birkh\"{a}user Boston, Inc., Boston, MA, 1985.

\bibitem[Bin53]{MR52763}
R.~H. Bing.
\newblock A convex metric with unique segments.
\newblock {\em Proc. Amer. Math. Soc.}, 4:167--174, 1953.

\bibitem[Bur15]{MR3358543}
Annegret~Y. Burtscher.
\newblock Length structures on manifolds with continuous {R}iemannian metrics.
\newblock {\em New York J. Math.}, 21:273--296, 2015.

\bibitem[CS21]{MR4232210}
Simone Cecchini and Thomas Schick.
\newblock Enlargeable metrics on nonspin manifolds.
\newblock {\em Proc. Amer. Math. Soc.}, 149(5):2199--2211, 2021.

\bibitem[Dav83]{MR690848}
Michael~W. Davis.
\newblock Groups generated by reflections and aspherical manifolds not covered
  by {E}uclidean space.
\newblock {\em Ann. of Math. (2)}, 117(2):293--324, 1983.

\bibitem[Den21]{MR4210894}
Jialong Deng.
\newblock Curvature-dimension condition meets {G}romov's {$n$}-volumic scalar
  curvature.
\newblock {\em SIGMA Symmetry Integrability Geom. Methods Appl.}, 17:013, 20
  pages, 2021.

\bibitem[DH89]{MR1000373}
Michael~W. Davis and Jean-Claude Hausmann.
\newblock Aspherical manifolds without smooth or {PL} structure.
\newblock In {\em Algebraic topology ({A}rcata, {CA}, 1986)}, volume 1370 of
  {\em Lecture Notes in Math.}, pages 135--142. Springer, Berlin, 1989.

\bibitem[DJ91]{MR1131435}
Michael~W. Davis and Tadeusz Januszkiewicz.
\newblock Hyperbolization of polyhedra.
\newblock {\em J. Differential Geom.}, 34(2):347--388, 1991.

\bibitem[DJL12]{MR2872552}
M.~Davis, T.~Januszkiewicz, and J.-F. Lafont.
\newblock {$4$}-dimensional locally {${\rm CAT}(0)$}-manifolds with no
  {R}iemannian smoothings.
\newblock {\em Duke Math. J.}, 161(1):1--28, 2012.

\bibitem[Fed69]{MR0257325}
Herbert Federer.
\newblock {\em Geometric measure theory}.
\newblock Die Grundlehren der mathematischen Wissenschaften, Band 153.
  Springer-Verlag New York Inc., New York, 1969.

\bibitem[FJ91]{MR1159252}
F.~Thomas Farrell and Lowell~E. Jones.
\newblock Rigidity in geometry and topology.
\newblock In {\em Proceedings of the {I}nternational {C}ongress of
  {M}athematicians, {V}ol. {I}, {II} ({K}yoto, 1990)}, pages 653--663. Math.
  Soc. Japan, Tokyo, 1991.

\bibitem[FJ93]{MR1216623}
F.~T. Farrell and L.~E. Jones.
\newblock Topological rigidity for compact non-positively curved manifolds.
\newblock In {\em Differential geometry: {R}iemannian geometry ({L}os
  {A}ngeles, {CA}, 1990)}, volume~54 of {\em Proc. Sympos. Pure Math.}, pages
  229--274. Amer. Math. Soc., Providence, RI, 1993.

\bibitem[GL80]{MR569070}
Mikhael Gromov and H.~Blaine Lawson, Jr.
\newblock Spin and scalar curvature in the presence of a fundamental group.
  {I}.
\newblock {\em Ann. of Math. (2)}, 111(2):209--230, 1980.

\bibitem[GL83]{MR720933}
Mikhael Gromov and H.~Blaine Lawson, Jr.
\newblock Positive scalar curvature and the {D}irac operator on complete
  {R}iemannian manifolds.
\newblock {\em Inst. Hautes \'{E}tudes Sci. Publ. Math.}, (58):83--196 (1984),
  1983.

\bibitem[Gro17]{Gromov}
Misha Gromov.
\newblock {101 Questions, Problems and Conjectures around Scalar Curvature},
  2017.
\newblock
  \url{https://www.ihes.fr/~gromov/wp-content/uploads/2018/08/101-problemsOct1-2017.pdf}.

\bibitem[Gro18]{MR3816521}
Misha Gromov.
\newblock Metric inequalities with scalar curvature.
\newblock {\em Geom. Funct. Anal.}, 28(3):645--726, 2018.

\bibitem[{Gro}19]{2019arXiv190810612G}
Misha {Gromov}.
\newblock {Four Lectures on Scalar Curvature}.
\newblock {\em arXiv e-prints}, page arXiv:1908.10612, Aug 2019.

\bibitem[Hit74]{MR358873}
Nigel Hitchin.
\newblock Harmonic spinors.
\newblock {\em Advances in Math.}, 14:1--55, 1974.

\bibitem[HS06]{MR2259056}
B.~Hanke and T.~Schick.
\newblock Enlargeability and index theory.
\newblock {\em J. Differential Geom.}, 74(2):293--320, 2006.

\bibitem[KS03]{MR1998480}
Gennadi Kasparov and Georges Skandalis.
\newblock Groups acting properly on ``bolic'' spaces and the {N}ovikov
  conjecture.
\newblock {\em Ann. of Math. (2)}, 158(1):165--206, 2003.

\bibitem[KW75]{MR365409}
Jerry~L. Kazdan and F.~W. Warner.
\newblock Scalar curvature and conformal deformation of {R}iemannian structure.
\newblock {\em J. Differential Geometry}, 10:113--134, 1975.

\bibitem[Lla98]{MR1600027}
Marcelo Llarull.
\newblock Sharp estimates and the {D}irac operator.
\newblock {\em Math. Ann.}, 310(1):55--71, 1998.

\bibitem[{Loh}18]{2018arXiv181211839L}
Joachim {Lohkamp}.
\newblock {Contracting Maps and Scalar Curvature}.
\newblock {\em arXiv e-prints}, page arXiv:1812.11839, Dec 2018.

\bibitem[LY72]{MR334083}
H.~Blaine Lawson, Jr. and Shing~Tung Yau.
\newblock Compact manifolds of nonpositive curvature.
\newblock {\em J. Differential Geometry}, 7:211--228, 1972.

\bibitem[Osa20]{MR4176066}
Damian Osajda.
\newblock Small cancellation labellings of some infinite graphs and
  applications.
\newblock {\em Acta Math.}, 225(1):159--191, 2020.

\bibitem[Sap14]{MR3110794}
Mark Sapir.
\newblock A {H}igman embedding preserving asphericity.
\newblock {\em J. Amer. Math. Soc.}, 27(1):1--42, 2014.

\bibitem[SY17]{2017arXiv170405490S}
Richard {Schoen} and Shing-Tung {Yau}.
\newblock {Positive Scalar Curvature and Minimal Hypersurface Singularities}.
\newblock {\em arXiv e-prints}, page arXiv:1704.05490, Apr 2017.

\bibitem[Ver18]{MR3884658}
Giona Veronelli.
\newblock Scalar curvature via local extent.
\newblock {\em Anal. Geom. Metr. Spaces}, 6(1):146--164, 2018.

\end{thebibliography}
\Addresses

\end{document}